\documentclass[reqno, 11pt]{amsart}
\usepackage{url}
\usepackage[colorlinks, linkcolor=blue, citecolor=blue, urlcolor=blue]{hyperref}
\usepackage{times}

\newtheorem{theorem}{Theorem}[section]
\newtheorem{lemma}[theorem]{Lemma}

\newtheorem{proposition}[theorem]{Proposition}
\newtheorem{corollary}[theorem]{Corollary}

\theoremstyle{definition}

\newtheorem{ex}[theorem]{Example}

\newtheorem{remark}[theorem]{Remark}

\numberwithin{equation}{section}

\usepackage{bbm}
\usepackage{url}
\usepackage{euscript}
\usepackage{amsmath}
\usepackage{amsthm}

\usepackage{amsxtra}
\usepackage{amssymb}
\usepackage{pifont}
\usepackage{amsbsy}

\usepackage{graphicx}
\usepackage{epstopdf}

\newskip\aline \newskip\halfaline
\def\skipaline{\vskip\aline}

\def\qedbox{$\rlap{$\sqcap$}\sqcup$}
\def\qed{\nobreak\hfill\penalty250 \hbox{}\nobreak\hfill\qedbox\skipaline}

\newcommand{\one}{\mathbbm{1}}

\newcommand\bR{{\mathbb R}}

\DeclareMathOperator{\tr}{{\rm tr}}
\DeclareMathOperator{\supp}{{\rm supp}}

\DeclareMathOperator{\spa}{span}

\DeclareMathOperator{\Cr}{\mathbf{Cr}}

\DeclareMathOperator{\Hess}{Hess}

\DeclareMathOperator{\var}{\boldsymbol{var}}
\DeclareMathOperator{\SO}{SO}
\DeclareMathOperator{\cov}{\boldsymbol{cov}}
\DeclareMathOperator{\Cov}{\boldsymbol{Cov}}
\DeclareMathOperator{\GOE}{GOE}

\newcommand{\be}{{\boldsymbol{e}}}

\newcommand{\bsF}{\boldsymbol{F}}

\newcommand{\bp}{{\boldsymbol{p}}}
\newcommand{\bq}{{\boldsymbol{q}}}

\newcommand{\bu}{{\boldsymbol{u}}}
\newcommand{\bv}{{\boldsymbol{v}}}

\newcommand{\bx}{{\boldsymbol{x}}}

\newcommand{\bsD}{\boldsymbol{D}}
\newcommand{\bsE}{\boldsymbol{E}}

\newcommand{\bsI}{\boldsymbol{I}}

\newcommand{\bsL}{\boldsymbol{L}}

\newcommand{\bsN}{\boldsymbol{N}}

\newcommand{\bsS}{\boldsymbol{S}}

\newcommand{\bsU}{{\boldsymbol{U}}}
\newcommand{\bsV}{\boldsymbol{V}}

\newcommand{\bsZ}{\boldsymbol{Z}}

\newcommand{\bgamma}{\boldsymbol{\gamma}}
\newcommand{\bGamma}{\boldsymbol{\Gamma}}

\newcommand{\bmu}{\boldsymbol{\mu}}

\newcommand{\bom}{{\boldsymbol{\omega}}}
\newcommand{\bsi}{\boldsymbol{\sigma}}
\newcommand{\bSi}{{\boldsymbol{\Sigma}}}

\newcommand{\si}{{\sigma}}

\newcommand{\ve}{{\varepsilon}}

\newcommand{\vfi}{{\varphi}}

\newcommand{\eE}{\EuScript{E}}

\newcommand{\eO}{\EuScript{O}}

\newcommand{\eR}{\EuScript{R}}
\newcommand{\eS}{\EuScript{S}}

\newcommand{\ra}{\rightarrow}

\newcommand{\Llra}{{\Longleftrightarrow}}

\newcommand{\lan}{\langle}
\newcommand{\ran}{\rangle}

\def\inpr{\mathbin{\hbox to 6pt{\vrule height0.4pt width5pt depth0pt \kern-.4pt \vrule height6pt width0.4pt depth0pt\hss}}}

\newcommand{\pa}{\partial}

\newcommand{\dual}{\spcheck{}}

\begin{document}

\title[Complexity of random smooth functions]{Complexity of random  smooth functions on compact manifolds.} 


\author{Liviu I. Nicolaescu}
\thanks{This work was partially supported by the NSF grant, DMS-1005745.}

\address{Department of Mathematics, University of Notre Dame, Notre Dame, IN 46556-4618.}
\email{nicolaescu.1@nd.edu}
\urladdr{\url{http://www.nd.edu/~lnicolae/}}


\begin{abstract} We     prove a universality   result relating  the  expected   distribution of  critical values  of  a random linear  combination  of eigenfunctions  of the Laplacian  on an arbitrary compact Riemann $m$-dimensional manifold   to  the expected distribution of eigenvalues of a  $(m+1)\times (m+1) $ random symmetric Wigner matrix. We then prove a central limit theorem  describing what happens to the expected distribution of critical values when the dimension of the manifold is very large. \end{abstract}

\maketitle

\tableofcontents

\section{Overview}
\setcounter{equation}{0}

The goal of this paper  is to  describe a  universal relationship between  the distribution of critical values  of certain random functions on an arbitrary  compact $m$-dimensional Riemann manifold and the distribution of eigenvalues  of  certain random symmetric $(m+1)\times(m+1)$-matrices.    A special case of this problem concerns the distribution of critical values of the restriction to the unit sphere $S^N\subset\bR^{N+1}$ of a random polynomial of very large degree in $(N + 1)$-variables.

\subsection{The setup}
Suppose that $(M,g)$ is a smooth, compact, connected  Riemann manifold of dimension $m>1$.   We denote by $|dV_g|$  the volume density     on $M$ induced by $g$. We assume that the metric is normalized so that
\[
{\rm vol}_g(M)=1.
\tag{$\ast$}
\label{tag: ast}
\]
For any $\bu, \bv\in C^\infty(M)$ we  denote by $(\bu,\bv)_g$ their $L^2$ inner product defined by  the metric $g$. The $L^2$-norm  of a smooth function $\bu$ is  denoted by $\|\bu\|$.

Let $\Delta_g: C^\infty(M)\ra C^\infty(M)$ denote the scalar Laplacian defined by the metric $g$.  For $L >0$ we set
\[
\bsU^L=\bsU^L (M,g):=\bigoplus_{\lambda\in [0, L^2]}\ker(\lambda-\Delta_g),\;\;d(L):=\dim \bsU_L.
\]
 We equip  $\bsU^L$  with the  Gaussian probability measure.
\[
d\bgamma^L(\bu):= (2\pi)^{-\frac{d(L)}{2}}e^{-\frac{\|\bu\|^2}{2}} |d\bu|.
\]
Fix  an orthonormal  Hilbert basis $(\Psi_k)_{k\geq 0}$ of $L^2(M)$ consisting of eigenfunctions of $\Delta_g$, 
\[
\Delta_g \Psi_k =\lambda_k \Psi_k,\;\; k_0\leq k_1 \Rightarrow \lambda_{k_0} \leq \lambda_{k_1}.
\]
Then  
\[
\bsU^L= \spa\bigl\{\,\Psi_k;\;\;\lambda_k\leq L^2\,\bigr\}.
\]
A random (with respect to $d\bgamma^L$)  function $\bu\in\bsU^L$    can be viewed as a linear combination
\[
\bu=\sum_{\lambda_k\leq L^2} u_k\Psi_k,
\]
where   $u_k$ are  i.i.d. Gaussian  random variables with mean $0$ and variance $\si^2=1$.    We have the following technical result whose proof is contained in Appendix \ref{s: morse}.

\begin{proposition} There exists $L_0>0$ such that for any $L\geq L_0$, a function  $\bu\in\bsU^L$ is  almost surely (a.s.) Morse.\qed
\label{prop: asmorse}
\end{proposition}

\begin{remark} For any $f\in C^\infty(M)$ and any open neighborhood $\eO$ of $f$ in $C^\infty(M)$ we can find $L_0\geq 0$ such that for any $L\geq L_0$ we have $\bsU^L\cap \eO\neq \emptyset$.  Suppose that  $f$ is \emph{stable}, i.e.,   $f$ is Morse  and the critical level sets  of $f$ contain  a single critical point.  If $\eO$ is sufficiently small, then any $f'\in \eO$ is topologically equivalent to $f$, \cite[Prop. III.2.2]{GG}. This means that there exists a diffeomorphism  $\Phi$ of $M$ and a diffeomorphism $\vfi$ of $R$ such that $f'=\vfi\circ f\circ\Psi^{-1}$.   Thus, as $L\to \infty$ the space $\bsU^L$ engulfs all the possible topological types of stable Morse functions.\qed
\end{remark}

For any $\bu\in C^1(M)$  we  denote by $\Cr(\bu)\subset M $ the set of critical points of $\bu$ and by $\bsD(\bu)$ the set of critical values\footnote{The set $\bsD(\bu)$ is sometime referred to as the \emph{discriminant set} of $\bu$.}  of $\bu$. If $L$ is sufficiently large the random set $\bsU^L\ni \bu\mapsto \Cr(\bu)$ is a.s. finite.  

To a Morse function  $\bu$ on $M$ we  associate  a Borel measure $\mu_\bu$ on $M$ and a Borel measure $\bsi_\bu$ on $\bR$   defined by  the equalities
\[
\mu_\bu: =\sum_{\bp\in \Cr(\bu)}  \delta_\bp,\;\;\bsi_\bu:=\bu_*(\mu_\bu)=\sum_{d\bu(\bp)=0} \delta_{\bu(\bp)}.
\]
Following the terminology on \cite{Auff0, Auff2}  we will refer to $\bsi_\bu$  as the \emph{variational complexity} of $\bu$. Observe that
\[
\supp\mu_\bu=\Cr(\bu),\;\;\supp\bsi_\bu=\bsD(\bu).
\]
When $\bu\in\bsU^L$     is not a  Morse function we  define $\mu_\bu$ and $\bsi_\bu$ arbitrarily.  We set
\begin{equation}
s_m:=\frac{(4\pi)^{-\frac{m}{2}}}{\Gamma(1+\frac{m}{2})},\;\;d_m:= \frac{(4\pi)^{-\frac{m}{2}} } {2  \Gamma(2+\frac{m}{2})},\;\;h_m:=\frac{(4\pi)^{-\frac{m}{2}} }{4 \Gamma(3+\frac{m}{2})}.
\label{eq: sdh}
\end{equation}
The  statistical significance of these numbers   is described is  Subsection \ref{s: main}.   We only want to mention here that the  H\"{o}rmander-Weyl spectral estimates  state that 
\begin{equation}
\dim \bsU^L=s_m L^m +O(L^{m-1})\;\;\mbox{as $L\ra \infty$}.
 \label{eq: HW}
 \end{equation}
For $L\gg 0$ , the correspondence $\bsU^L\ni \bu \mapsto \mu_\bu$ is a random measure on $M$ called the \emph{empirical distribution of  critical points} of the  random function. Its expectation is the measure  $\mu^L$ on $M$  defined by
\[
\int_M   f d\mu^L =\int_{\bsU^L} \left( \int_M  f d\mu_\bu\right)  d\bgamma^L(\bu),
\]
for any continuous  function $f:M\ra \bR$. Note that the number   
\[
\bsN^L :=\int_M    d\mu^L = \int_{\bsU^L} |\Cr(\bu)| d\bgamma^L(\bu)
\]
is  the expected number of critical points  of a random function  in $\bsU^L$. 

In  \cite{N2} we showed that there exists a universal (explicit) constant $C_m$ that depends only on the dimension $m$ such that
\begin{equation}
\bsN^L\sim C_m\dim\bsU^L\sim C_m s_m^\bom(L)^m \;\;\mbox{as $L\ra \infty$},
\label{eq: NL}
\end{equation}
 and the normalized measures 
 \[
 d\bar{\mu}^L:=\frac{1}{\bsN^L} d\mu^L
 \]
 converges weakly to the      metric volume  measure $|dV_g|$ as $L\to\infty$.  This means   that  for $L$ very large  we expect the critical set  of a random $\bu\in \bsU^L$ to be   close to uniformly distributed on $M$.  
 
 Similarly, the random measure $\bsU^L\ni\bu\mapsto \bsi_\bu$  has  an expectation $\bsi^L:=\bsE_{\bsU^L}(\bsi_\bu)$ which is a finite measure on $\bR$ defined by
\[
\int_\bR f(\lambda) d\bsi^L(\lambda)=\int_{\bsU^L} \left(\int_{\bR} f(\lambda)d\bsi_\bu^L(\lambda)\right)d\bgamma^L(\bu),
\]
for any continuous and bounded function $f:\bR\to\bR$.      Results of  Adler-Taylor \cite{AT} (see Subection \ref{s: KR}) show that  $\bsi^L$ exists.  Note that
\[
\int_{\bR}\bsi^L(dt)=\bsN^L.
\]

 \subsection{Statements of the main results} In this  paper we investigate the statistical properties of the  measure $\bsi^L$  as $L\ra \infty$ and then as $m\to\infty$.    To state our results we need  a bit of terminology.
 
 For any $t>0$ we denote by $\eR_t:\bR\ra \bR$ the rescaling map $\bR\ni x\mapsto tx\in \bR$. If $\mu$ is a Borel measure on $\bR$ we denote by $(\eR_t)_*\mu$ its pushforward via the rescaling map $\eR_t$.   We denote by $\bgamma_v$ the Gaussian measure on $\bR$ with mean zero and variance $v\geq 0$. 
 
 The central result of this paper  states that  as $L\to\infty$ the  probability measures 
 \[
\frac{1}{\bsN^L}\Bigl(\eR_{\frac{1}{\sqrt{\dim\bsU^L}}}\Bigr)_*\bsi^L
\]
converge weakly to a  probability  measure $\bsi_m$ on $\bR$  which can be  described  explicitly in terms of the statistics of the the eigenvalues of    certain  random symmetric $(m+1)\times (m+1)$-matrices. Additionally we  prove a central limit theorem stating  that as $m\to \infty$, the  probability measures $\bsi_m$ converge weakly to a Gaussian measure $\bgamma_2$.      Before we give a more precise   description of the measure $\bsi_m$ we need to  point out   a small annoyance which we will turn to our advantage.
 
 Observe that if $\bu:M\to\bR$ is a  fixed Morse function  and $c$ is a constant,  then 
 \[
 \Cr(c+\bu)=\Cr(\bu), \;\;\mu_{c+\bu}=\mu_\bu, 
 \]
  but   
\[
\bsD(\bu+c)=c+\bsD(\bu),\;\;\bsi_{\bu+c}=\delta_c\ast \bsi_\bu,
\]
where $\ast$ denotes the convolution of two  finite measures on $\bR$.  More generally, if $X$  is a scalar random variable   with  probability distribution  $\nu_X$, then    the  expected variational complexity  of the random function $X+\bu$ is the  measure  $\nu_X\ast \bsi_u$, where $\ast$ denotes the operation of convolution. In    particular, if the distribution $\nu_X$ is a Gaussian,  then the  measure   $\bsi_\bu$ is uniquely determined by the  measure   $\nu_X\ast \bsi_u$ since the convolution with a Gaussian is an injective operation.  

It turns out that  it is easier to understand the statistics of the variational  complexity  of the perturbation of a random  $\bu\in\bsU_L$ by an independent  Gaussian variable of  cleverly chosen variance.  
 
   We      consider random functions of the form
 \[
 \bu_\bom =X_\bom+\bu=X_\bom+\sum_{\lambda_k\leq L^2} u_k\Psi_k,
 \]
 where  the   Fourier coefficients  $ u_k$ are i.i.d. standard Gaussians, and   $X_\bom\in\bsN(0,\bom)$ is a   scalar random variable  independent of the $u_k$'s.    In applications $\bom$ will depend on $m$ and $L$.     
 
 Since $X_\bom$ is independent of $\bu$  we deduce that the  expected variational complexity of $X_\bom+\bu$  is the measure $\bsi_\bom^L$ on $\bR$  given by
 \begin{equation}
 \bsi_\bom^L=\bsE(\,\bsi_{X_\bom+\bu})=\gamma_\bom\ast \bsi^L.
 \label{eq: bsib}
 \end{equation}
Note that
 \[
 \bsN^L=\int_\bR   d\bsi_\bom^L(t)=\int_\bR d\bsi^L(t).
 \]
 The  first goal  of this  paper is to investigate the behavior of the probability measures  $ \frac{1}{\bsN^L}\bsi_\bom^L$ as $L\ra \infty$ for  certain  very special $\bom$'s.  For reasons that  we will shortly,  we   choose $\bom$   of  the form  
 \begin{equation}
 \bom=\bom(m, L, r)= \bar{\bom}_{m,r}L^m
 \label{eq: bom1}
 \end{equation}
 where  $r>0$ and the  positive   quantity $\bar{\bom}_{m,r}$ are uniquely determined by the equality
 \begin{equation}
  s_m+\bar{\bom}_{m,r} = r\frac{d_m^2}{h_m}=: s_m^\bom.
  \label{eq: bom2}
  \end{equation}
  Observe that as $L\to \infty$  we have $\bom(m,L,r)\to \infty$ so  the random variable $X_\bom$ is more and more diffuse. From  the elementary identity
  \begin{equation}
s_m=h_m(m+2)(m+4),\;\;d_m=(m+4) h_m
\label{eq: sdh1}
\end{equation}
  we deduce that
  \begin{equation}
  s_m^\bom=r\frac{m+4}{m+2} s_m,\;\; \bar{\bom}_{m,r}=\left( \frac{r(m+4)}{m+2)}-1\right)s_m.
  \label{eq: bom3}
  \end{equation}
  The  inequality $s_m^\bom\geq s_m$ shows  that the parameter $r$  must satisfy the $m$-dependent constraint
\begin{equation*}
r\geq \frac{m+2}{m+4}.
\tag{$C_m$}
\label{tag: rcon}
\end{equation*}

For $v\in (0,\infty)$ and $N$ a positive integer we denote by $\GOE_N^v$ the space $\eS_N$  of real, symmetric  $N\times N$  matrices $A$ equipped with a Gaussian  measure such that the entries $a_{ij}$ are independent, zero-mean,  normal   random variables with variances
\[
\var(a_{ii})=2v,\;\;\var(a_{ij})=v,\;\;\forall 1\leq i <j\leq N.
\]
We  denote by $\rho_{N,v}(\lambda)$  the  \emph{normalized correlation} function  of $\GOE^v_N$. It is uniquely determined  by the equality
\[
\int_{\bR} f(\lambda) \rho_{N,v}(\lambda) d\lambda=\frac{1}{N}\bsE_{\GOE_N^v}\bigl(  \tr f(A)\,\bigr),
\]
for any  continuous  bounded function $f:\bR\ra \bR$.  The function $\rho_{N,v}(\lambda)$  also has a probabilistic interpretation.  For any Borel set $B\subset \bR$  the expected number  of eigenvalues in $B$ of  a random $A\in \GOE_N^v$ is equal to
\[
N\int_B\rho_{N,v}(\lambda) d\lambda.
\]
 The celebrated Wigner semicircle theorem, \cite[Thm. 2.1.1]{AGZ}, \cite[Eq.(7.2.33)]{Me},    states that  as $N\ra \infty$ the rescaled probability measures 
\[
\bigl(\,\eR_{\frac{1}{\sqrt{N}}}\,\bigr)_*\bigl(\,\rho_{N,v}(\lambda)d\lambda\,\bigr)
\]
converge  weakly to the  semicircle measure given by the density
\[
\rho_{\infty,v}(\lambda):=\frac{1}{2\pi v}\times \begin{cases}
\sqrt{4v-\lambda^2}, &|\lambda|\leq \sqrt{4v}\\
0, & |\lambda| >\sqrt{4v}.
\end{cases}
\]
We can now explain  what we gain by working with the perturbed function $\bu_\bom= X_\bom+\bu$.    The computation of  $\bsi^L$ uses the (conditioned)   Gaussian random matrix
\begin{equation}
Z^{L,x}=\Bigl(\,L^{-\frac{m+4}{2}}\Hess(\bu_\bom,\bp)\,\Bigr|\, d\bu_\bom(\bp)=0,\;\;\bu_\bom(\bp)=L^{\frac{m}{2}}x\,\Bigr),
\label{eq: zlc0}
\end{equation}
where $\Hess(\bu_\bom,\bp)$  denotes the hessian of $\bu_\bom$  at $\bp$ and $x$ is a fixed real number. The probability distribution of   this  random matrix    depends on the choice of  $\bom$.

 In Lemma \ref{lemma: zlc} we show that if $\bom$ is chosen according to the prescriptions (\ref{eq: bom1}), (\ref{eq: bom3})  where $r\geq 1$, then the random $m\times m$  matrix  (\ref{eq: zlc0})     converges as $L\to \infty$ to  a sum between a $\GOE_m^{h_m}$-distributed  matrix and  an independent normally distributed  multiple  of the identity; see especially (\ref{eq: zlc1}).   Once this happens we can  use a  simple trick of   Fyodorov \cite{Fy} to  express the limit  as $L\to \infty$ of the expectation of $|\det Z^{L,x}|$      in terms of statistics  of the  ensemble $\GOE^{h_m}_{m+1}$;  see Lemmas \ref{lemma: exp-det}, \ref{lemma: exp-det1}.  For example, in the extreme case      $r=1$,   the random matrix $Z^{L,x}+\frac{x}{m+4}\one_m$  converges    as $L\to \infty$ to a random matrix in the  ensemble  $\GOE_m^{h_m}$.       
 
 None of the above  nice accidents would   take place if we did not perturb the function  by the carefully chosen   random variable $X_\bom$.    The choice  of $X_\bom$  is essentially forced upon us by  (\ref{eq: zlc1}) and the discussion in  Appendix \ref{s: gmat}, especially the equality (\ref{eq: smgoe}).

We can now state   the  main technical  result of this paper.

\begin{theorem} Fix a positive real number satisfying  $r\geq 1$.     Let $\bom=\bom(m,L,r)$ be  defined by the equalities (\ref{eq: bom1}) and (\ref{eq: bom2}). Then as $L\ra \infty$ the     rescaled  measures
\[
\frac{1}{\bsN^L}\Bigl(\eR_{\frac{1}{\sqrt{s_m^\bom L^m}}} \Bigr)_*\bsi_\bom^L
\]
converge  weakly     to a probability  measure   $\bsi_{m,r}$ on $\bR$ satisfying  the  equality
\begin{equation}
\bsi_{m,r}\propto \gamma_{\frac{(r-1)}{r}}\ast\bigl(\, e^{-\frac{r\lambda^2}{4}}\rho_{m+1, r^{-1}}(\lambda) d\lambda\,\bigr),
\label{eq: lim-si}
\end{equation}
where $\propto$ denotes the relation of proportionality of two finite measures. In particular, when $r=1$ we have
\[
\bsi_{m,1}\propto e^{-\frac{\lambda^2}{4}}\rho_{m+1,1}(\lambda)d\lambda.
\]
\label{th: main}
\end{theorem}

The above result has several interesting consequences. 

\begin{corollary}[Universality.]   As $L\ra \infty$  the rescaled measures 
\[
\frac{1}{\bsN^L}\Bigl(\eR_{\frac{1}{\sqrt{\dim\bsU^L}}}\Bigr)_*\bsi^L
\]
converge weakly to a  probability  measure $\bsi_m$ on $\bR$ uniquely determined by the convolution equation
\[
\bgamma_{\frac{2}{m+2}}\ast\bsi_m=\Bigl(\eR_{\sqrt{\frac{m+4}{m+2}}} \Bigr)_*\bsi_{m,1}.
\]
\label{cor: main2}
\end{corollary}

 The fact that the convolution equation
\[
\bgamma_{\frac{2}{m+2}}\ast\mu=\Bigl(\eR_{\sqrt{\frac{m+4}{m+2}}} \Bigr)_*\bsi_{m,1} 
\]
has a  \emph{unique} solution $\mu$ can be seen easily  by taking Fourier transforms. The above result shows that the  large $L$ behavior of the   average complexity $\bsi^L$ is independent of the background manifold $M$ and of the metric $g$.   We  do not  have a more explicit  and simpler description of $\bsi_m$ and we doubt   that such  a description exists.

 \begin{corollary} As $m\ra \infty$, the  measures $\bsi_{m}$ converge weakly to the   Gaussian  measure $\bgamma_2$. 
  \label{cor: main3}
 \end{corollary}

 \subsection{Related results} The scaling limit of various statistical  quantities associated to Gaussian random fields  in the high frequency limit has  been studied in detail  over the past    fifteen years.   In \cite{Zel}  S. Zelditch  investigates  the volume of  the zero set  of  such a  random field and proves a related universality result.  The zero set  of a  random section of a large power  of   an ample  line bundle over a K\"{a}hler manifold displays a similar  universal scaling behavior; see \cite{BSZ} and the references therein.
 
 The distribution of critical points  (or energy landscape)  of  isotropic random functions on $\bR^m$ was investigated by Fyodorov \cite{Fy, Fy1} who also relates this problem to the staistics of the eigenvalues in the ensemble $\GOE_{m+1}$.  Recently A. Auffinger \cite{Auff0,Auff2} has investigated the distributions of critical  values  of  certain isotropic random fields on a   round sphere $S^m$, where $m\to\infty$,  and   described  a connection with   the distribution of eigenvalues   of symmetric matrices  in the ensemble  $\GOE_{m+1}$.
 
  The  scaling limit of the distribution of critical points of random  holomorphic sections  of a large power of an ample line bundle on a K\"{a}hler manifold was investigated  by  M. Douglas, B. Schiffman and S. Zelditch, \cite{DSZ1, DSZ2, DSZ3}. The dimensional dependence   of the number of critical points of such a random holomorphic section was described by B. Baugher, \cite{Bau}.

The universality  result described in Theorem \ref{th: main}  is not an isolated phenomenon and   fits a more general pattern.  To explain this,   fix a  measurable function $w:[0,\infty)\to [0,\infty)$ called \emph{weight}. For $L>0$ we define the rescaled  weight
\[
w_L:[0,\infty)\to[0,\infty),\;\;w_L(t)=w\Bigl(\frac{t}{L}\Bigr)
\]
and consider the random function  on $M$
\begin{equation}
\bu^L=\sum_{k\geq 0}\sqrt{w_L\bigr({\lambda_k}^{\frac{1}{2}}\bigr)} C_k\Psi_k,
\label{eq: ran3}
\end{equation}
where $C_k$ are  independent normally distributed  random variables.  If $w$ decays sufficiently fast  as $t\to \infty$, then $\bu^L$ is a.s. smooth.   

The correlation kernel  of the random function (\ref{eq: ran3}) can be identified  with the Schwartz kernel of the smoothing operator $w_L(\sqrt{\Delta})$.  We can then ask   about the behavior as $L\to\infty$  of the expected variational complexity of the random function (\ref{eq: ran3}).  This is in turn conditioned by the regularity  of $w$: the more regular  is $w$ the more detailed  is the information about this behavior.    This paper covers the "worst" situation, when $w=\bsI_{[0,1]}$ is obviously discontinuous. Above and throughout  this paper we use the notation $\bsI_A$ to denote the indicator function of a subset $A$ of a given set $S$.

In \cite{N3},  a sequel to this paper,  we  investigate    the random  functions  defined by weights  $w$ which are smooth.   We show  that Theorem \ref{th: main} has a  suitable counterpart   in this case.  Moreover,  the smoothness of $w$  allows us to provide   additional  nontrivial information that is  not available in the  singular case $w=\bsI_{[0,1]}$.

 \subsection{Organization of the paper} Let us briefly  describe   the principles hiding behind the above results. Theorem \ref{th: main}  follows from a Kac-Rice type formula \cite{AT, DSZ1} aided by the  refined spectral estimates of the spectral function of the Laplacian on a compact Riemann manifold,  \cite{Bin,  DV, Hspec, SV}.    Corollary \ref{cor: main2}  is a rather immediate consequence of Theorem \ref{th: main} while Corollary    \ref{cor: main3}  follows   from  Corollary \ref{cor: main2} via  a refined version of Wigner's semicircle theorem.
  
  The basic facts   coverning  the Kac-Rice  formula are presented in   Subsection \ref{s: KR} while the proofs  of the above results are presented in Subsections \ref{s: main}, \ref{s: main2}, \ref{s: main3}.     Appendix \ref{s: morse} is devoted to the proof of Proposition \ref{prop: asmorse}.  To aid the  reader with a more geometric bias we have included two probabilistic appendices.  In  Appendix \ref{s: gauss} we have collected a few  basic facts about Gaussian measures used throughout the paper.   In  the more exotic Appendix \ref{s: gmat} we discuss  a family of    symmetric random matrices   and some of their properties needed in the main  body of the paper.

 \medskip
 
 \subsection*{Aknowledgments.}    I  would like to thank the anonymous referee for the many constructive  suggestions    that   helped me improve the quality of the present paper.
 
 \section{Proofs}
 \label{s: pro}
\setcounter{equation}{0} 
 
\subsection{A Kac-Rice type formula}
\label{s: KR}

As we have already mentioned, the key  result  behind   Theorem \ref{th: main} is a Kac-Rice type result which we intend to discuss in some detail in this  section.    This  result gives an explicit, yet  quite complicated description of the measure  $\bsi_\bom^L$. More precisely, for any Borel subset $B\subset \bR$ the Kac-Rice formula provides an integral representation of $\bsi_\bom^L(B)$ of the form
\[
\bsi_\bom^L(B)=\int_M f_{L, \bom, B}(\bp)\,|dV_g(\bp)|,
\]
for some integrable function $f_{L,\bom, B}: M\ra \bR$. The core of the Kac-Rice  formula is an explicit  probabilistic description of the density $f_{L, \bom, B}$.

Fix a point $\bp\in M$.      This determines three   Gaussian random variables.
\begin{equation*}
\tag{$RV_\bom$}
\begin{split}
(\bsU^L,\bgamma_\bom^L)\ni \bu_\bom\mapsto \bu_\bom (\bp)\in\bR,\\
(\bsU^L,\bgamma_\bom^L)\ni \bu_\bom\mapsto  d\bu_\bom (\bp)\in T^*_\bp M,\\
(\bsU^L,\bgamma_\bom^L)\ni \bu_\bom\mapsto \Hess_\bp(\bu_\bom)\in \eS(T_\bp M),
\end{split}
\label{eq: rvmain}
\end{equation*}
where $\Hess_\bp(\bu_\bom): T_\bp M\times T_\bp M\ra \bR$ is the Hessian of $\bu_\bom$ at $\bp$ defined in terms  of the Levi-Civita  connection of $g$ and then identified with a symmetric endomorphism of $T_\bp M$  using  again the metric $g$.     More concretely, if $(x^i)_{1\leq i\leq m}$ are $g$-normal coordinates at $\bp$, then
\[
\Hess_\bp(\bu_\bom)\pa_{x^j}=\sum_{i=1}^m\pa^2_{x^ix^j}\bu_\bom(\bp)\pa_{x^i}.
\]
As shown in the proof of Proposition \ref{prop: asmorse}, for $L$  very large the map $\bsU^L\ni\bu\mapsto  d\bu(\bp)\in T^*_\bp M$ is surjective  which implies that the covariance form of the  Gaussian random vector $d\bu_\bom(\bp)$  is positive definite. We can identify it with a  symmetric, positive definite   linear operator
\[
\bsS\bigl(\, d\bu_\bom(\bp)\,\bigr) : T_\bp M\ra T_\bp M.
\]
More  concretely, if $(x^i)_{1\leq i\leq m}$ are  $g$-normal coordinates  at $\bp$, then we can  identify $\bsS\bigl(\, d\bu_\bom(\bp)\,\bigr)$ with a $m\times m$ real symmetric  matrix whose $(i,j)$-entry is given by
\[
\bsS_{ij}\bigl(\, d\bu_\bom(\bp)\,\bigr)=\bsE\bigl(\, \pa_{x_i}\bu_\bom(\bp)\cdot \pa_{x^j}\bu_\bom(\bp)\,\bigr).
\]
 \begin{theorem}  Fix a Borel subset $B\subset \bR$. For any $\bp\in M$ define
 \[
 f_{L,\bom, B}(\bp):=\left(\,\det\bigl(\, 2\pi\bsS(\,\bu_\bom(\bp)\,\,\bigr)\,\right)^{-\frac{1}{2}}\bsE\Bigl(\; |\det \Hess_\bp(\bu_\bom)|\,\cdot\, \bsI_B(\,\bu_\bom(\bp)\,)\;\;\bigl|\; d\bu_\bom(\bp)=0\;\Bigr),
 \]
where  $\bsE\bigl( \;\mathbf{var}\;\; | \; \mathbf{cons}\;\bigr)$ stands for the conditional expectation  of the variable $\mathbf{var}$ given the constraint $\mathbf{cons}$.  Then 
 \begin{equation}
 \bsi_\bom^L(B)=\int_M f_{L,\bom, B}(\bp)\,|dV_g(\bp)|.
 \label{eq: KR}
 \end{equation}
 \qed
 \label{th: KR}
 \end{theorem}
 
 This theorem is a special case of a general result of Adler-Taylor,  \cite[Thm. 11.2.1]{AT}. The   many   technical assumptions in  Adler-Taylor Theorem are trivially  satisfied in this case.   In  \cite{N2}  we proved this theorem in the case $B=\bR$ and $\bom=0$.  The strategy used there can be  modified to  yield the more general Theorem \ref{th: KR}.

For the above theorem  to be of any use we need to have some concrete  information about the Gaussian  random variables (\ref{eq:  rvmain}).    All the relevant statistical invariants of these  variables can be extracted from the covariance kernel of the random function $\bu_\bom$. This  is the function 
\[
\eE_\bom^L:M\times  M\to\bR,
\]
\[
\eE_\bom^L (\bp,\bq)=\bsE\bigl( \bu_\bom(\bp)\bu_\bom(\bq)\,\bigr)= \bsE\bigl(  (X+\bu(\bp)\,)\,\cdot\, (X+\bu(\bq)\,)\,\bigr)
\]
\[
=\bom+\sum_{\lambda_k\leq L^2}\Psi_k(\bp)\Psi_k(\bq)=:\bom+\eE^L(\bp,\bq).
\]
The function  $\eE^L$  is   the spectral function of the Laplacian, i.e., the  Schwartz kernel of $P_L$, the orthogonal projection onto $\bsU^L$. Fortunately, a lot is known  about the behavior of $\eE^L$ as $L\ra \infty$,  \cite{Bin, DV, Hspec, SV, Zel}.

\subsection{Proof of  Theorem \ref{th: main}}
\label{s: main}

Fix $L\gg 0$. For any $\bp\in M$ we have the centered  Gaussian   vector   (\ref{eq: rvmain}), $\bom=0$,
\[
(\bsU^L,\bgamma^L)\ni u\mapsto \bigl(\,\bu(\bp), d\bu(\bp), \Hess_\bp(\bu)\,\bigr)\in\bR\oplus T_\bp^*M\oplus \eS(T_\bp M).
\]
We fix normal coordinates $(x^i)_{1\leq i\leq m}$ at $\bp$  and we can identify the above Gaussian vector  with the centered Gaussian vector
\[
\bigl(\,\bu(\bp),\; \bigl(\,\pa_{x^i}\bu(p)\,\bigr)_{1\leq i\leq m},\;\bigl(\,\pa^2_{x^ix^j}\bu(\bp)\,\bigr)_{1\leq i,j\leq m}\,\bigr)\in\bR\oplus \bR^m\oplus \eS_m.
\]
In \cite[\S 3]{N2}  we showed that the spectral estimates of  Bin-H\"{o}rmander  \cite{Bin, Hspec}  imply the following asymptotic estimates.

\begin{lemma}  For any $1\leq i,j,k,\ell\leq m$ we have the  uniform in $\bp$ asymptotic estimates as $L\ra \infty$
\begin{subequations}
\begin{equation}
\bsE(\,\bu(p)^2\,\bigr)= s_m^\bom L^m\bigl(1+O(L^{-1})\,\bigr),
\label{eq: cov0-0}
\end{equation}
\begin{equation}\bsE\bigl(\, \pa_{x^i}\bu(\bp)\pa_{x^j}\bu(\bp)\,\bigr)= d_mL^{m+2}\delta_{ij}\bigl(\,1+O(L^{-1})\,\bigr),
\label{eq: cov1-0}
\end{equation}
\begin{equation}
\bsE\bigl(\, \pa^2_{x^ix^j}\bu(\bp)\pa^2_{x^kx^\ell}\bu(\bp)\,\bigr)= h_mL^{m+4}(\delta_{ij}\delta_{k\ell}+\delta_{ik}\delta_{j\ell}+ \delta_{i\ell}\delta_{jk})\bigl(\,1+O(L^{-1})\,\bigr),
\label{eq: cov2-0}
\end{equation}
\begin{equation}
\bsE\bigl(\,\bu(\bp)\pa^2_{x^ix^j}\bu(\bp)\,\,\bigr)= -d_mL^{m+2}\delta_{ij}\bigl(\,1+O(L^{-1})\,\bigr),
\label{eq: cov3-0}
\end{equation}
\begin{equation}
\bsE\bigl(\,\bu(p)\pa_{x^i}\bu(\bp)\,\bigr)= O(L^m),\;\; \bsE\bigl(\,\pa_{x^i}\bu(\bp)\pa^2_{x^jx^k}\bu(\bp)\,\bigr)=O(L^{m+2}),
\label{eq: cov4-0}
\end{equation}
\end{subequations}
where the constants $s_m, d_m, h_m$ are defined  by (\ref{eq: sdh}). \qed
\label{lemma: asymp}
\end{lemma}

Now let $\bom=\bom(m,L, r)$ be defined as in (\ref{eq: bom1}), (\ref{eq: bom2}). Using the notation (\ref{eq: bom3}) we deduce from the above that   in the case of the random  function $\bu_\bom$ we have the estimates

\begin{subequations}
\begin{equation}
\bsE(\,\bu_\bom(p)^2\,\bigr)= s^\bom_mL^m\bigl(1+O(L^{-1})\,\bigr),
\label{eq: cov0}
\end{equation}
\begin{equation}\bsE\bigl(\, \pa_{x^i}\bu_\bom(\bp)\pa_{x^j}\bu_\bom(\bp)\,\bigr)= d_mL^{m+2}\delta_{ij}\bigl(\,1+O(L^{-1})\,\bigr),
\label{eq: cov1}
\end{equation}
\begin{equation}
\bsE\bigl(\, \pa^2_{x^ix^j}\bu_\bom(\bp)\pa^2_{x^kx^\ell}\bu_\bom(\bp)\,\bigr)= h_mL^{m+4}(\delta_{ij}\delta_{k\ell}+\delta_{ik}\delta_{j\ell}+ \delta_{i\ell}\delta_{jk})\bigl(\,1+O(L^{-1})\,\bigr),
\label{eq: cov2}
\end{equation}
\begin{equation}
\bsE\bigl(\,\bu_\bom(\bp)\pa^2_{x^ix^j}\bu_\bom(\bp)\,\,\bigr)= -d_mL^{m+2}\delta_{ij}\bigl(\,1+O(L^{-1})\,\bigr),
\label{eq: cov3}
\end{equation}
\begin{equation}
\bsE\bigl(\,\bu_\bom(p)\pa_{x^i}\bu_\bom(\bp)\,\bigr)= O(L^m),\;\; \bsE\bigl(\,\pa_{x^i}\bu_\bom(\bp)\pa^2_{x^jx^k}\bu_\bom(\bp)\,\bigr)=O(L^{m+2}).
\label{eq: cov4}
\end{equation}
\end{subequations}
From the estimate (\ref{eq: cov1})  we deduce that
\[
\bsS(\, d\bu_\bom(\bp)\,) =d_mL^{m+2}\bigl(\,\one_m+O(L^{-1})\,\bigr),
\]
so that
\begin{equation}
\sqrt{|\det \bsS(\bu_\bom(p))|}= (d_m)^{\frac{m}{2}} L^{\frac{m(m+2)}{2}} \bigl(\,1+O(L^{-1})\,\bigr)\;\;\mbox{as $L\ra\infty$}.
\label{eq: asy-det}
\end{equation}
Consider the rescaled random vector
\[
\begin{split}
(s^L, v^L, H^L)=&(s^{L,\bom, \bp}, v^{L,\bom,. \bp}, H^{L,\bom, \bp})\\
& := \bigl(L^{-\frac{m}{2}} \bu_\bom(\bp) , L^{-\frac{m+2}{2}} d\bu_\bom(p),\;\; L^{-\frac{m+4}{2}} \Hess_\bp \bu_\bom\,\bigr).
\end{split}
\]
Form the above we deduce  the following uniform in $\bp$ estimates as $L\ra \infty$.
\begin{subequations}
\begin{equation}
\bsE(\,(s^L)^2 \,\bigr)= s^\omega_m\bigl(1+O(L^{-1})\,\bigr),
\label{eq: cov0r}
\end{equation}
\begin{equation}\bsE\bigl(\, v_i^Lv_j^L\,\bigr)= d_m\delta_{ij}\bigl(\,1+O(L^{-1})\,\bigr),
\label{eq: cov1r}
\end{equation}
\begin{equation}
\bsE\bigl(\, H^L_{ij}H^L_{kl}\,\bigr)= h_m(\delta_{ij}\delta_{k\ell}+\delta_{ik}\delta_{j\ell}+ \delta_{i\ell}\delta_{jk})\bigl(\,1+O(L^{-1})\,\bigr),
\label{eq: cov2r}
\end{equation}
\begin{equation}
\bsE\bigl(\,s^LH_{ij}^L\,\,\bigr)= -d_m\delta_{ij}\bigl(\,1+O(L^{-1})\,\bigr),
\label{eq: cov3r}
\end{equation}
\begin{equation}
\bsE\bigl(\,s^Lv^L_i\,\bigr)= O(L^{-1}),\;\; \bsE\bigl(\,v^L_iH^L_{jk}\,\bigr)=O(L^{-1}).
\label{eq: cov4r}
\end{equation}
\end{subequations}

The probability distribution of the variable $s^L$ is 
\[
d\gamma_{s^L}(x)=\frac{1}{\sqrt{2\pi \bar{s}_m^\bom(L)}}e^{-\frac{x^2}{2\bar{s}_m^\bom(L)}} |dx|,
\]
where  its variance $\bar{s}_m^\bom(L)$ satisfies the estimate
\[
\bar{s}_m^\bom(L)=s^\bom_m+O(L^{-1}).
\]
Fix a Borel set $B\subset\bR$. We have
\begin{equation}
\begin{split}
\bsE\bigl(\, |\det\Hess \bu_\bom(\bp)| \bsI_{ B}\bigl(\,\bu_\bom(\bp)\,\bigr)\;\bigl| \; d\bu_\bom(\bp)=0\,\bigr)  =L^{\frac{m(m+4)}{2}}\bsE\bigl(\, |\det H^L| \bsI_{L^{-\frac{m}{2}}B}(s^L)\;\bigl| \; v^L=0\,\bigr)& \\
=L^{\frac{m(m+4)}{2}}\underbrace{\int_{L^{-\frac{m}{2}}B} \bsE\bigl(\, |\det H^L| \;\bigl|\; s^L=x,\;v^L=0\,\bigr)  \frac{e^{-\frac{x^2}{2\bar{s}_m^\bom(L)}}}{\sqrt{2\pi \bar{s}_m^\bom(L)}} |dx|}_{=:q_{L,\bp}(L^{-\frac{m}{2}}B)}.
\end{split}
\label{eq: asy-exp}
\end{equation}
Using (\ref{eq: asy-det}) and (\ref{eq: asy-exp}) we deduce from Theorem \ref{th: KR} that
\[
\bsi_\bom^L(B) = L^m\left(\frac{1}{2\pi d_m}\right)^{\frac{m}{2}}\int_M q_{L,\bp}(L^{-\frac{m}{2}} B) \rho_L(\bp) |dV_g(\bp)|,
\]
where   $\rho_L:M\ra \bR$ is a function that satisfies the uniform in $\bp$ estimate 
\begin{equation}
\rho_L(p)=1+O(L^{-1})\;\;\mbox{as $L\ra \infty$}.
\label{eq: rhol}
\end{equation}
Hence
\begin{equation}
\frac{1}{L^m}\left(\eR_{ L^{-\frac{m}{2}}}\right)_* \bsi_\bom^L(B)=  \left(\frac{1}{2\pi d_m}\right)^{\frac{m}{2}}\int_M q_{L,\bp}(B) \rho_L(\bp) |dV_g(\bp)|.
\label{eq: dlb}
\end{equation}
To continue the  computation we need to investigate the behavior of $q_{L,\bp}(B)$ as $L\ra \infty$. More concretely,  we need  to elucidate the nature of the Gaussian  matrix
\[
Z^{L,x}:=\bigl(\, H^L\; \bigl|\; s^L=x,\;v^L=0\,\bigr).
\]
\begin{lemma}  Set  
\[
\kappa=\kappa(r):=\frac{(r-1)}{2r},\;\;r\geq 1. 
\]
The  Gaussian random matrix $Z^{L,x}$  converges uniformly in $\bp$ as $L\ra  \infty$ to the random matrix $A-\frac{x}{r(m+4)}\one_m$, where    $A$  belongs to the Gaussian ensemble $\eS_m^{2\kappa h_m, h_m}$ described in   Appendix \ref{s: gmat}.
\label{lemma: zlc}
\end{lemma}

\begin{proof} We will use the regression formula (\ref{eq: cov-regr}). For simplicity we set
\[
Y^L:= (s^L, v^L)\in\bR\oplus \bR^m.
\]
The components of $Y$ are   
\[
Y^L_0=s^L,\;\; Y^L_i=v^L_i,\;\;1\leq i\leq m.
\]
Using (\ref{eq: cov0r}), (\ref{eq: cov1r}) and (\ref{eq: cov4r}) we deduce that for any $1\leq i, j\leq m$ we have
\[
\bsE(Y^L_0Y_i^L)=s^\bom_m\delta_{0i}+O(L^{-1}),\;\;\bsE(Y^L_iY_j^L)= d_m\delta_{ij}+O(L^{-1}).
\]
If $\bsS(Y^L)$ denotes the covariance operator of $Y^L$,  then we deduce that
\begin{equation}
\bsS(Y^L)^{-1}_{0,i}= \frac{1}{s^\bom_m}\delta_{0i}+O(L^{-1}),\;\;\bsS(Y^L)^{-1}_{ij}=\frac{1}{d_m}\delta_{ij}+O(L^{-1}). 
\label{eq: cov5}
\end{equation}
We now need to compute the covariance operator $\Cov(H^L,Y^L)$. To do so  we equip $\eS_m$ with the inner product
\[
(A,B)=\tr(AB),\;\;A,B\in\eS_m
\]
The space    $\eS_m$ has a canonical orthonormal basis $\widehat{E}_{ij}$, $1\leq i\leq j\leq m$, where 
\[
\widehat{\bsE}_{ij}= \begin{cases}
\bsE_{ij}, & i=j\\
\frac{1}{\sqrt{2}} \bsE_{ij}, & i<j
\end{cases}
\]
and   $\bsE_{ij}$ denotes the symmetric matrix     nonzero entries only at locations $(i,j)$ and $(j,i)$ and these entries are equal to $1$.  Thus a matrix $A\in\eS_m$ can be written as
\[
A=\sum_{i\leq j} a_{ij} \bsE_{ij} =\sum_{i\leq j} \hat{a}_{ij} \widehat{\bsE}_{ij},
\]
where
\[
\widehat{a}_{ij}=  \begin{cases}
a_{ij}, &i=j,\\
\sqrt{2} a_{ij}, &i<j.
\end{cases}
\]
The covariance operator  $\Cov(H^L,Y^L)$ is a linear map
\[
\Cov(H^L,Y^L):\bR\oplus \bR^m \ra \eS_m
\]
given by
\[
\Cov(H^L,Y^L)\left(\sum_{\alpha=0}^m y_\alpha \be_\alpha \right) = \sum_{i<j, \alpha}  \bsE(\widehat{H}_{ij}^L  Y_\alpha^L)y_\alpha\widehat{\bsE}_{ij}=\sum_{i<j, \alpha}  \bsE({H}^L_{ij}  Y^L_\alpha)y_\alpha{\bsE}_{ij},
\]
where $\be_0,\be_1,\dotsc,\be_m$ denotes the canonical orthonormal basis in $\bR\oplus\bR^m$.  Using (\ref{eq: cov3r}) and (\ref{eq: cov4r}) we deduce  that
\begin{equation}
\Cov(H^L,Y^L)\left(\sum_{\alpha=0}^m y_\alpha \be_\alpha \right) = -y_0d_m \one_m+ O(L^{-1}).
\label{eq: cov6}
\end{equation}
We deduce that  the transpose $\Cov(H^\ve,Y\ve)\dual$ satisfies
\begin{equation}
\Cov(H^L,Y^L)\dual \left(\sum_{i\leq j} \hat{a}_{ij}\widehat{\bsE}_{ij}\right)=  -d_m\tr(A)\be_0+O(L^{-1}).
\label{eq: cov6d}
\end{equation}
The covariance operator of the random symmetric matrix  $Z^L=XZ^{L,x}$  is then
\[
\bsS(Z^L)= \bsS(H^L) - \Cov(H^L,Y^L)\bsS(Y^L)^{-1} \Cov(H^L,Y^L)\dual.
\]
This means that
\[
\bsE\bigl(\, \widehat{z}_{ij}^L\cdot \widehat{z}_{k\ell}^L\,\bigr)= (\widehat{\bsE}_{ij}, \bsS(Z^L)\widehat{\bsE}_{k\ell}).
\]
Using (\ref{eq: cov5}), (\ref{eq: cov6}) and (\ref{eq: cov6d}) we deduce that
\[
\Cov(H^L,Y^L)\bsS(Y^L)^{-1} \Cov(H^L,Y^L)\dual\left(\sum_{i\leq j} \hat{a}_{ij}\widehat{\bsE}_{ij}\right)=\frac{d_m^2}{s^\bom_m}\tr(A)\one_m+O(L^{-1})
\]
\[
\bsE\bigl(\,(z_{ij}^L)^2\,\bigr)= h_m+ O(L^{-1}), \;\;\bsE(z_{ii}^Lz_{jj}^L)=h_m-\frac{d_m^2}{s_m^\bom}+ O(L^{-1}),\;\;\forall i<j,
\]
\[
\bsE\bigr(\,(z_{ii}^L)^2\,\bigr)=3h_m-\frac{d_m^2}{s_m^\bom}+O(L^{-1}),\;\; \forall i,
\]
and 
\[
\bsE(z_{ij}^Lz_{k\ell}^L)=O(L^{-1}),\;\;\forall i<j,\;\;k\leq \ell,\;\;(i,j)\neq (k,\ell).
\]
We can rewrite these equalities in the compact form
\begin{equation}
\bsE(z_{ij}^Lz_{k\ell}^L)=\left(h_m-\frac{d_m^2}{s_m^\bom}\right)\delta_{ij}\delta_{k\ell} + h_m(\delta_{ik}\delta_{j\ell}+ \delta_{i\ell}\delta_{jk})+O(L^{-1}).
\label{eq: zlc1}
\end{equation}
The last equality is ultimately the  main reason for our choices (\ref{eq: bom1}) and (\ref{eq: bom3}) in  defining $X_\bom$.  Note that with $r$ defined as in (\ref{eq: bom2}) we have
\[
h_m-\frac{d_m^2}{s^\bom_m}\stackrel{(\ref{eq: sdh1})}{=} \frac{r-1}{r}h_m.
\]
Hence
\begin{equation}
\bsE(z_{ij}^Lz_{k\ell}^L)=2\kappa h_m\delta_{ij}\delta_{k\ell} + h_m(\delta_{ik}\delta_{j\ell}+ \delta_{i\ell}\delta_{jk})+O(L^{-1}).
\label{eq: zlc2}
\end{equation}
Using (\ref{eq: cond-expect})  we deduce that the expectation of $Z^L$ is 
\begin{equation}
\bsE(Z^L)=\Cov(H^L,Y^L)\bsS(Y^L)^{-1}(x\be_0)= -\frac{x}{r(m+4)}\one_m+O(L^{-1}).
\label{eq: expect}
\end{equation}
This  completes the proof of Lemma \ref{lemma: zlc}.
\end{proof} 
We deduce that 
\[
\lim_{L\ra \infty}q_{L,\bp}(B) =q_\infty(B):=\int_B \bsE_{\eS_m^{2\kappa h_m, h_m}}\bigl(\, \bigl|\,\det\bigl(\,  A-\frac{x}{r(m+4)} \one_m\,\bigr)\,\bigr|\,\bigr)\frac{ e^{- \frac{x^2}{2s^\bom_m}}}{\sqrt{2\pi s^\bom_m}} dx
\]
\[
=(h_m)^{\frac{m}{2}} \int_B \bsE_{\eS_m^{2\kappa, 1}}\bigl(\, \bigl|\,\det\bigl(\,  A-\frac{x}{r(m+4)\sqrt{h_m}} \one_m\,\bigr)\,\bigr|\,\bigr)\frac{ e^{- \frac{x^2}{2s^\bom_m}}}{\sqrt{2\pi s^\bom_m}} dx
\]
\[
=(h_m)^{\frac{m}{2}} \int_{(s_m^\bom)^{-\frac{1}{2}}B} \bsE_{\eS_m^{2\kappa, 1}}\bigl(\, \bigl|\,\det\bigl(\,  A-\alpha_m y \one_m\,\bigr)\,\bigr|\,\bigr)\frac{ e^{- \frac{y^2}{2}}}{\sqrt{2\pi }} dy,
\]
where
\[
\alpha_m=\frac{\sqrt{s_m^\bom}}{r(m+4)\sqrt{h_m}}\stackrel{(\ref{eq: sdh1})}{=}\frac{1}{\sqrt{r}}.
\]
This proves that
\[
\lim_{L\ra \infty}\left(\, \eR_{(s_m^\bom)^{-\frac{1}{2}}}\,\right)_*q_{L,\bp}(B)=(h_m)^{\frac{m}{2}} \underbrace{\int_{B} \bsE_{\eS_m^{2\kappa, 1}}\Bigl(\, \Bigl|\,\det\Bigl(\,  A-\frac{y}{\sqrt{r}} \one_m\,\Bigr)\,\Bigr|\,\Bigr)\frac{ e^{- \frac{y^2}{2}}}{\sqrt{2\pi }} dy}_{=:\mu_m(B)}.
\]
Using the last equality, the normalization (\ref{tag: ast}) and the estimate  (\ref{eq: rhol}) in  (\ref{eq: dlb}) we conclude 
\begin{equation}
\begin{split}
\lim_{L\ra \infty} \frac{1}{s_m^\bom L^m}\bigl(\, \eR_{ (s_m^\bom L^m)^{ -\frac{1}{2}}}\,\bigr)_*\bsi_\bom^L(B)=\frac{1}{s_m}\left(\frac{h_m}{2\pi d_m}\right)^{\frac{m}{2}}\mu_m(B)\\
\stackrel{(\ref{eq: sdh1})}{=}\left(\frac{2}{m+4}\right)^{\frac{m}{2}}\Gamma\left(1+\frac{m}{2}\right)\mu_m(B).
\end{split}
\label{eq: limit}
\end{equation}
In particular, this shows that
\[
\bsN^L \sim s_m^\bom L^m \left(\frac{2}{m+4}\right)^{\frac{m}{2}}\Gamma\left(1+\frac{m}{2}\right)\mu_m(\bR).
\]
Observe that the probability density of $\mu_m$ is
\begin{equation}
\frac{d\mu_m}{dy}= \bsE_{\eS_m^{2\kappa, 1}}\Bigl(\, \Bigl|\,\det\Bigl(\,  A-\frac{y}{\sqrt{r}} \one_m\,\Bigr)\,\Bigr|\,\Bigr)\frac{ e^{- \frac{y^2}{2}}}{\sqrt{2\pi }}.
\label{eq: dmu}
\end{equation}
We now distinguish two cases.

\medskip

\noindent {\bf Case 1.} $r>1$ From Lemma \ref{lemma: exp-det1} we deduce that
\begin{equation}
\begin{split}
   \bsE_{\eS_m^{2\kappa, 1}}\Bigl(\, \Bigl|\,\det\Bigl(\,  A-\frac{y}{\sqrt{r}} \one_m\,\Bigr)\,\Bigr|\,\Bigr) \;\;\;\;\;\;\;\;\; \;\;\;\;\;\;\;\;\;\;\;& \\
  =2^{\frac{m+3}{2}}\Gamma\left(\frac{m+3}{2}\right)\frac{1}{\sqrt{2\pi\kappa}}\int_{\bR}\rho_{m+1,1}(\lambda)  e^{-\frac{1}{4\tau^2}(\lambda-(\tau^2+1)\frac{y}{\sqrt{r}})^2 +\frac{(\tau^2+1)y^2}{4r}} d\lambda  &  ,
\end{split}
\label{eq: exp-cond}
\end{equation}
where
\[
\tau^2:=\frac{\kappa}{\kappa-1}=\frac{r-1}{r+1}.
\]
Thus
\[
\frac{d\mu_m}{dy}= 2^{\frac{m+3}{2}}\Gamma\left(\frac{m+3}{2}\right)\frac{1}{2\pi\sqrt{\kappa}} e^{\frac{(\tau^2+1-2r)y^2}{4r}}  \int_{\bR}\rho_{m+1,1}(\lambda)  e^{-\frac{1}{4\tau^2}(\lambda-(\tau^2+1)\frac{y}{\sqrt{r}})^2 } d\lambda
\]
\[
= 2^{\frac{m+3}{2}}\Gamma\left(\frac{m+3}{2}\right)\frac{1}{2\pi\sqrt{\kappa}} \int_{\bR}\rho_{m+1,1}(\lambda)  e^{-\frac{1}{4\tau^2}(\lambda-(\tau^2+1)\frac{y}{\sqrt{r}})^2 -\frac{ry^2}{2(r+1)}} d\lambda.
\]
An elementary computation  shows that
\[
-\frac{1}{4\tau^2}\left(\lambda-(\tau^2+1)\frac{y}{\sqrt{r}}\right)^2 -\frac{ry^2}{2(r+1)}= -\frac{1}{4}\lambda^2-\left( \sqrt{\frac{1}{2(r-1)}} \lambda -y \sqrt{\frac{r}{2(r-1)}}\right)^2.
\]
Now set
\[
\beta=\beta(r):= \frac{1}{(r-1)} .
\]
We deduce
\[
\frac{d\mu_m}{dy}=  2^{\frac{m+3}{2}}\Gamma\left(\frac{m+3}{2}\right)\frac{1}{2\pi\sqrt{\kappa}} \int_{\bR}\rho_{m+1,1}(\lambda) e^{-\frac{1}{4}\lambda^2} e^{-\frac{\beta}{2}(\lambda-\sqrt{r}y)^2} d\lambda
\]
($\lambda:=\sqrt{r}\lambda$)
\[
= 2^{\frac{m+3}{2}}\Gamma\left(\frac{m+3}{2}\right)\frac{1}{2\pi\sqrt{\kappa}}  \int_{\bR}\sqrt{r}\rho_{m+1,1}(\sqrt{r}\lambda) e^{-\frac{r}{4}\lambda^2} e^{-\frac{r\beta}{2}(\lambda-y)^2} d\lambda
\]
\[
\stackrel{(\ref{eq: resc-cor})}{=}2^{\frac{m+3}{2}}\Gamma\left(\frac{m+3}{2}\right)\frac{1}{\sqrt{2\pi\kappa r\beta}}  \int_{\bR}\rho_{m+1,1/r}(\lambda) e^{-\frac{r}{4}\lambda^2} d\gamma_{\frac{1}{\beta r}}(y-\lambda) d\lambda.
\]
Using the last equality in  (\ref{eq: limit})   and  then  invoking the estimate (\ref{eq: NL}) we obtain the case $r>1$ of Theorem \ref{th: main}.

\medskip

\noindent {\bf Case 2.} $r=1$.  The  proof of  Theorem \ref{th: main} in  this case follows a similar pattern.  Note first that in this case $\kappa=0$ so  invoking Lemma \ref{lemma: exp-det} we obtain the  following counterpart of (\ref{eq: exp-cond}) 
\[
 \bsE_{\GOE_m^1}\Bigl(\, \Bigl|\,\det\Bigl(\,  A-y \one_m\,\Bigr)\,\Bigr|\,\Bigr)=2^{\frac{m+4}{2}}\Gamma\left(\frac{m+3}{2}\right) e^{\frac{y^2}{4}}\rho_{m+1,1}(y).
 \]
Using this in (\ref{eq: dmu}) we deduce immediately (\ref{eq:  lim-si}) in the case $r=1$.  This completes the proof of Theorem \ref{th: main}.

\qed

 \subsection{Proof of Corollary \ref{cor: main2}.} \label{s: main2}  We use  Theorem \ref{th: main} in the case $r=1$.  Using (\ref{eq:  bsib}), (\ref{eq: bom1}) and (\ref{eq: bom3}) we deduce that  when $r=1$ we have
 \begin{equation}
 s_m =s_m^\bom \frac{m+2}{m+4},\;\; \bsi_\bom^L= \gamma_{\frac{2s_m L^m}{m+2}} \ast\bsi^L.
 \label{eq: bom0}
 \end{equation}
 We deduce
 \begin{equation}
\frac{1}{\bsN^L} \Bigl( \eR_{\frac{1}{\sqrt{s_m L^m}}}\Bigr)_*\bsi_\bom^L= \Bigl(\eR_{\sqrt{\frac{m+4}{m+2}}} \Bigr)_*\left(\,\frac{1}{\bsN^L}\Bigl( \eR_{\frac{1}{\sqrt{s_m^\bom  L^m}}}\Bigr)_*\bsi_\bom^L\right).
 \label{eq:  resc}
 \end{equation}
 Using  (\ref{eq: bom0}) we deduce that 
 \begin{equation}
\frac{1}{\bsN^L} \Bigl( \eR_{\frac{1}{\sqrt{s_m L^m}}}\Bigr)_*\bsi_\bom^L= \gamma_{\frac{2}{m+2}}\ast\left(\, \frac{1}{\bsN^L} \Bigl( \eR_{\frac{1}{\sqrt{s_m L^m}}}\Bigr)_*\bsi^L\,\right).
\label{eq: resc1}
\end{equation}
Using  the  spectral estimates (\ref{eq: HW}), the equality (\ref{eq: resc})  and  Theorem \ref{th: main} with $r=1$ we deduce
\[
 \lim_{L\ra \infty} \gamma_{\frac{2}{m+2}}\ast\left(\frac{1}{\bsN^L} \Bigl( \eR_{\frac{1}{\sqrt{s_mL^m}}}\Bigr)_*\bsi^L\right)=\lim_{L\ra \infty} \frac{1}{\bsN^L} \Bigl( \eR_{\frac{1}{\sqrt{s_m^\bom L^m}}}\Bigr)_*\bsi_\bom^L= \Bigl(\eR_{\sqrt{\frac{m+4}{m+2}}} \Bigr)_*\bsi_{m+1,1}.
 \]
We can now conclude by invoking L\'{e}vy's continuity theorem \cite[Thm. 15.23(ii)]{Kl}. Here are the details.

Denote by $ {\mu}(\xi)$  and respectively ${\mu}^L_\bom(\xi)$ the Fourier transforms of the measures
\[
\frac{1}{\bsN^L}\Bigl( \eR_{\frac{1}{\sqrt{s_m^\bom L^m}}}\Bigr)_*\bsi_\bom^L\;\;\mbox{and respectively}\;\;  \frac{1}{\bsN^L} \Bigl( \eR_{\frac{1}{\sqrt{s_m L^m}}}\Bigr)_*\bsi^L
\]
Observe that the Fourier transform of the Gaussian measure $\gamma_{\frac{2}{m+2}}$ is $e^{-\frac{1}{(m+2)}|\xi|^2}$. Then (\ref{eq: resc1}) implies
 \begin{equation}
\mu^L(\xi)=e^{\frac{1}{(m+2)}|\xi|^2}\mu^L_\bom(\xi).
\label{eq: resc2}
\end{equation}
Theorem \ref{th: main} coupled with Levy's theorem  imply that the  family of  functions  $\mu^L_\bom(\xi)$ has a limit  $\mu^\infty_\bom(\xi)$ as $L\to \infty$. Hence the family $\mu^L(\xi)$ has a limit $\mu^\infty(\xi)$ as $L\to\infty$ satisfying
\[
\mu^\infty(\xi)=e^{\frac{1}{(m+2)}|\xi|^2}\mu^\infty_\bom(\xi).
\]
The limit $\mu^\infty_\bom(\xi)$ is the Fourier transform of 
\[
\Bigl(\eR_{\sqrt{\frac{m+4}{m+2}}} \Bigr)_*\bsi_{m+1,1}.
\]
Invoking  Levy's theorem   again,  we deduce from (\ref{eq: resc2}) that   the measures  
\[
\frac{1}{\bsN^L} \Bigl( \eR_{\frac{1}{\sqrt{s_m^\bom(L)}}}\Bigr)_*\bsi^L
\]
 converge as $L\to \infty$ to a measure    $\bsi_m$ whose Fourier transform is $\mu^L(\xi)$.   The equality  
\[
\bgamma_{\frac{2}{m+2}}\ast\bsi_m=\Bigl(\eR_{\sqrt{\frac{m+4}{m+2}}} \Bigr)_*\bsi_{m,1},
\]
 is the Fourier inverse of the equality (\ref{eq: resc2}).
\qed

\subsection{Proof of Corollary \ref{cor: main3}.}\label{s: main3}      By invoking Levy's continuity theorem and  Corollary \ref{cor: main2} we see that is suffices to show that the  probability measures  $\bsi_{m,1}$ converge  weakly to the  Gaussian measure $\gamma_2$.
 
  Set
 \[
 \bar{R}_m(x):=\sqrt{m}\rho_{m+1,1}(\sqrt{m} \, x)=\rho_{m+1, \frac{1}{m}}(x),
 \]
\[
R_\infty(x)=\frac{1}{2\pi} \bsI_{\{|x|\leq 2\}}\sqrt{4-x^2}. 
\]
Fix $c\in (0,2)$.  In \cite[\S 4.2]{N2}. we proved that
\begin{subequations}
\begin{equation}
\lim_{m\ra \infty}\sup_{|x|\leq c} |\bar{R}_m(x)-R_\infty(x)|=0,
\label{eq: est2}
\end{equation}
\begin{equation}
\sup_{|x|\geq c} |\bar{R}_m(x)-R_\infty(x)|= O(1)\;\;\mbox{as $m\to\infty$}.
\label{eq: est3}
\end{equation}
\end{subequations}
We deduce that
\begin{equation}
\rho_{m+1,1}(\lambda)e^{-\frac{\lambda^2}{4}}= \sqrt{\frac{4\pi}{{m}}}\bar{R}_m\left(\frac{\lambda}{\sqrt{m}}\right)\frac{1}{\sqrt{4\pi}}e^{-\frac{\lambda^2}{4}},
\label{eq: repr}
\end{equation}
 and  
\[
I_m:=\int_\bR \rho_{m+1,1}(\lambda)e^{-\frac{\lambda^2}{4}} d\lambda= \sqrt{\frac{4\pi}{{m}}}\int_\bR \bar{R}_m(x)\sqrt{\frac{m}{4\pi}} e^{-\frac{mx^2}{4}} dx=  \sqrt{\frac{4\pi}{{m}}}\int_\bR \bar{R}_m(x) d\gamma_{\frac{2}{m}}(x).
\]
The estimates  (\ref{eq: est2}), (\ref{eq: est3}) imply that
\[
I_m \sim \sqrt{4\pi}R_\infty(0) m^{-\frac{1}{2}} \;\;\mbox{as $m\ra \infty$}.
\]
To prove that the probability measures
\[
\frac{1}{I_m}\rho_{m+1,1}(\lambda) e^{-\frac{\lambda^2}{4}}  d\lambda
\]
converges weakly  to  $\bgamma_2$ it suffices to show that the finite measures
\[
\nu_m:=m^{\frac{1}{2}}  \rho_{m+1,1}(\lambda) e^{-\frac{\lambda^2}{4}}  d\lambda
\]
converge weakly to   the finite measure  $\nu_\infty:=R_\infty(0)e^{-\frac{\lambda^2}{4}} d\lambda$. 
  
Let $f: \bR\ra \bR$ be a bounded continuous function. Using (\ref{eq: repr}) we deduce  that 
\[
\int_\bR f(\lambda) d\nu_m(\lambda)= \int_\bR f(\lambda)\bar{R}_m\bigl(m^{-\frac{1}{2}} \lambda\bigr ) e^{-\frac{\lambda^2}{4}} d\lambda.
\]
We deduce that
\[
\int_\bR f(\lambda) d\nu_m(\lambda)- \int_\bR f(\lambda) d\nu_\infty(\lambda)= \int_\bR f(\lambda)\left(\, \bar{R}_m\bigl(m^{-\frac{1}{2}} x\bigr)- R_\infty(0) \right)e^{-\frac{\lambda^2}{4}} d\lambda
\]
\[
=\underbrace{\int_\bR f(\lambda)\bsI_{\{|\lambda|\leq c\sqrt{m}\}}\left(\, \bar{R}_m\bigl(m^{-\frac{1}{2}} x\bigr)- R_\infty(0) \right)e^{-\frac{\lambda^2}{4}} d\lambda}_{A_m} 
\]
\[
+ \underbrace{\int_\bR f(\lambda)\bsI_{\{|\lambda|\geq c\sqrt{m}\}}\left(\, \bar{R}_m\bigl(m^{-\frac{1}{2}} x\bigr)- R_\infty(0) \right)e^{-\frac{\lambda^2}{4}} d\lambda}_{B_m}.
\]
The estimate (\ref{eq: est2}) coupled with the  dominated convergence theorem imply that $A_m\to 0$ as $m\to \infty$. The estimate  (\ref{eq: est3}) and the dominated convergence theorem  imply that $B_m\to\infty$ as $m\to\infty$. \qed

\appendix
\section{Proof of Proposition \ref{prop: asmorse}}
\label{s: morse}

We will prove that there exists $L_0>0$ such that for any $L\geq L_0$ the space $\bsU^L$ is \emph{ample}, i.e., for any $\bp\in M$ and any $\xi\in T^*_\bp M$ there exists $\bu\in\bsU^L$ such that  $d\bu(\bp)=\xi$.   We can then invoke \cite[Cor. 1.26]{N1} to conclude that  the  functions in  $\bsU^L$ are a.s. Morse.

Choose     smooth functions $f_1,\dotsc, f_N: M\to \bR$ such that the map
\[
M\ni \bp\mapsto \bigl(\, f_1(\bp),\dotsc, f_N(\bp)\,\bigr)\in\bR^N
\]
is a smooth embedding.  Denote by $\bsF$ the subspace of $C^\infty(M)$ spanned by the functions   $f_1,\dotsc, f_N$  and by $P_L: L^2(M)\to \bsU^L$ the $L^2$-orthogonal projection  onto $\bsU^L$.
\begin{lemma}  
\[
\lim_{L\to\infty} \sup_{f\in \bsF\setminus 0} \frac{\|f-P_Lf\|_{C^2}}{\|f\|_{C^2}}=0.
\]
\label{lemma: appro}
\end{lemma}

\begin{proof} Fix a basis $\vfi_1,\dotsc, \vfi_\nu$ of $\bsF$, $\nu=\dim\bsF$ so that any $f\in \bsF$  has a unique decomposition
\[
f=\sum_{i=1}^\nu x_i(f)\vfi_i,\;\;x_i(f)\in\bR.
\]
Since $\dim \bsF<\infty$ the  $C^2$-norm on $\bsF$ is   equivalent with the norm
\[
\|f\|^*:=\sum_{i=1}^\nu |x_i(f)|.
\]
We have
\[
\|f-P_L f\|_{C^2}  \leq\sum_{i=1}^\nu  |x_i(f)|\|\vfi_i-P_L\vfi_i\|_{C^2}\leq \|f\|^* \max_{1\leq i\leq \nu} \|\vfi_i-P_L \vfi_i\|_{C^2}
\]
\[
\leq C\Bigl(\max_{1\leq i\leq \nu} \|\vfi_i-P_L \vfi_i\|_{C^2}\Bigr)\|f\|_{C^2},
\]
for some constant $C>0$.  Now observe that
\[
\max_{1\leq i\leq \nu} \|\vfi_i-P_L \vfi_i\|_{C^2}\to 0\;\;\mbox{as}\;\;L\to \infty.
\]
\end{proof}

To prove the ampleness of $\bsU^L$ for $L$ large  we argue by contradiction. Thus, we assume  that for any  positive integer  $n$ we can find $\bp_n\in M$ and a tangent vector $X_n\in T_{\bp_n}M$ such that
\[
|X_n|_g=1,\;\; d\bu(X_n)=0,\;\;\forall u\in\bsU^n.
\]
Upon extracting a subsequence we can assume that $\bp_n\to \bp_\infty$ and $X_n\to X_\infty\in T_{\bp_\infty} M$ as $n\to \infty$.   Since the space $\bsF$ is obviously ample we can find $f_\infty\in \bsF$ such that $d f_\infty (X_\infty)=1$.  Set $\bu_n:=P_n f_\infty$. Then  $d\bu_n(X_n)=0$ for any $n$ and
\[
 |df_\infty(X_n)|  =\bigl|\,d(\,f_\infty(X_n)- \bu_n(X_n)\,)\,\bigr|\leq \|f_\infty-P_nf_\infty\|_{C^2}\leq \ve_n \|f_\infty\|_{C^2},
 \]
 where $\ve_n\to 0$ as $n\to \infty$ according to Lemma \ref{lemma: appro}. On the other  hand
 \[
 df_\infty(X_n)\to df_\infty(X_\infty)=1.
 \]
 This contradiction completes the proof of Proposition \ref{prop: asmorse}. \qed

\section{Gaussian measures and Gaussian vectors}
\label{s: gauss}
\setcounter{equation}{0}

For the reader's convenience we  survey here a few  basic facts about Gaussian  measures. For more details we refer to \cite{Bog}.  A   \emph{Gaussian measure} on $\bR$  is a Borel measure $\bgamma_{\mu,v}$, $v\geq 0$, $m\in\bR$,   of the form
\[
\bgamma_{\mu,v}(dx)= \frac{1}{\sqrt{2\pi v}} e^{-\frac{ (x-\mu)^2}{2v}} dx.
\]
The scalar $\mu$ is called the \emph{mean}, while $v$ is called the \emph{variance}. We allow $v$ to be zero in which case
\[
\bgamma_{\mu,0}=\delta_\mu=\mbox{the Dirac measure on $\bR$ concentrated at $\mu$}.
\]
 For a  real valued random variable  $X$ we write
 \begin{equation}
 X\in \bsN(\mu, v)
 \label{eq: normal}
 \end{equation}
 if the probability   measure of $X$ is $\bgamma_{\mu, v}$.

 Suppose that $\bsV$ is a finite dimensional vector space with dual $\bsV\dual$. A \emph{Gaussian measure} on $\bsV$ is a  Borel measure  $\gamma$ on $\bsV$ such that, for any $\xi\in\bsV\dual$, the pushforward $\xi_*(\gamma)$ is a Gaussian measure on $\bR$, 
\[
\xi_*(\gamma)=\bgamma_{\mu(\xi),\si(\xi)}.
\]
One can show that  the map $\bsV\dual\ni \xi\mapsto \mu(\xi)\in\bR$ is linear, and  thus can be identified with a  vector $\bmu_\gamma\in \bsV$ called the \emph{barycenter} or \emph{expectation} of $\gamma$ that can be alternatively defined by the equality 
\[
\bmu_\gamma=\int_{\bsV} \bv d\gamma(\bv). 
\]
Moreover, there exists  a   nonnegative definite, symmetric bilinear map
\[
\bSi: \bsV\dual\times\bsV\dual\ra \bR\;\;\mbox{such that}\;\;\si(\xi)^2= \bSi(\xi,\xi),\;\;\forall \xi\in \bsV\dual.
\]
The form  $\bSi$ is called the \emph{covariance form} and can be identified with a  linear operator $\bsS:\bsV\dual\ra \bsV$ such that
\[
\bSi(\xi,\eta)=\lan \xi, \bsS\eta\ran,\;\;\forall \xi,\eta\in \bsV\dual,
\]
where $\lan-,-\ran:\bsV\dual\times\bsV\ra \bR$ denotes the natural bilinear pairing between a vector space and its dual. The operator $\bsS$ is called the \emph{covariance operator} and it is explicitly described by the  integral formula
\[
\lan \xi, \bsS\eta\ran=\Lambda(\xi,\eta)=\int_{\bsV}\lan \xi,\bv-\bmu_\gamma\ran \lan\eta, \bv-\bmu_\gamma\ran d\gamma(\bv).
\]
The Gaussian measure is said to be \emph{nondegenerate} if $\bSi$ is nondegenerate, and it is called \emph{centered} if $\bmu=0$. A nondegenerate    Gaussian measure  on $\bsV$ is uniquely determined by  its covariance form and its  barycenter. 

\begin{ex}  Suppose that   $\bsU$ is an $n$-dimensional  Euclidean space with  inner product $(-,-)$. We use the inner product to identify $\bsU$ with its dual $\bsU\dual$. If $A:\bsU\ra \bsU$ is a symmetric, positive definite   operator,  then
\begin{equation}
d\bgamma_A(\bx) =\frac{1}{(2\pi)^{\frac{n}{2}}\sqrt{\det A}} e^{-\frac{1}{2}(A^{-1}\bu, \bu)}\,|d\bu|
\label{eq: gamaA}
\end{equation}
is a  centered  Gaussian  measure on $\bsU$ with covariance form described by the  operator $A$.\qed
\end{ex}

If  $\bsV$ is a  finite dimensional vector space  equipped with a  Gaussian measure  $\gamma$  and $\bsL: \bsV\ra \bsU$ is a   linear  map, then the pushforward $\bsL_*\gamma$ is a  Gaussian measure on   $\bsU$ with barycenter
\[
\bmu_{\bsL_*\gamma}=\bsL(\bmu_\gamma)
\]
and covariance form
\[
\bSi_{\bsL_*\gamma}:\bsU\dual\times \bsU\dual\ra \bR,\;\; \bSi_{\bsL_*\gamma}(\eta,\eta)= \bSi_\gamma(\bsL\dual\eta,\bsL\dual\eta),\;\;\forall \eta \in \bsU\dual,
\]
where $\bsL\dual:\bsU\dual\ra \bsV\dual$ is the dual (transpose) of the linear map $\bsL$. Observe that if $\gamma$ is nondegenerate and $\bsL$ is surjective, then $\bsL_*\gamma$ is also nondegenerate.

Suppose $(\eS, \mu)$ is a probability space.    A \emph{Gaussian} random vector on $(\eS,\mu)$ is a (Borel) measurable map
\[
X: \eS\ra \bsV,\;\;\mbox{$\bsV$ finite dimensional vector space}
\]
such that $X_*\mu$ is a Gaussian measure on $\bsV$. We will refer to this measure as the \emph{associated Gaussian measure}, we denote it  by $\gamma_X$ and  we denote by $\bSi_X$ (respectively $\bsS(X)$) its covariance form (respectively operator),
\[
\bSi_X(\xi_1,\xi_2)=\bsE\bigl(\, \lan \xi_1, X-\bsE(X)\,\ran\,\lan \xi_2, X-\bsE(X)\,\ran\,\bigr).
\]
 Note that the   expectation of $\gamma_X$ is precisely the expectation of $X$. The random vector is called \emph{nondegenerate}, respectively \emph{centered}, if the Gaussian measure  $\gamma_X$ is such.
 
 Let us point out that if $X:\eS\ra \bsU$ is a Gaussian random vector and $\bsL:\bsU\ra \bsV$ is a linear map, then the random vector $\bsL X:\eS\ra \bsV$ is also Gaussian. Moreover
\[
\bsE(\bsL X)= \bsL \bsE(X),\;\;\bSi_{\bsL X}(\xi,\xi)=\bSi_X(\bsL\dual\xi,\bsL\dual\xi),\;\;\forall \xi\in\bsV\dual,
\]
where $\bsL\dual: \bsV\dual\ra \bsU\dual$ is the linear map dual to $\bsL$. Equivalently,  $\bsS(\bsL X)= \bsL \bsS(X )\bsL\dual$.

 Suppose  that $X_j:\eS\ra \bsV_1$, $j=1,2$, are two \emph{centered} Gaussian random vectors such that the direct sum $X_1\oplus X_2:\eS\ra \bsV_1\oplus \bsV_2
$ is also a centered  Gaussian random vector with  associated  Gaussian measure
\[
\gamma_{X_1\oplus X_2}= p_{X_1\oplus X_2} (\bx_1,\bx_2) |d\bx_1d\bx_2|.
\]
 We obtain a bilinear form
\[
\cov(X_1,X_2):\bsV_1\dual\times \bsV_2\dual\ra \bR,\;\; \cov(X_1,X_2)(\xi_1,\xi_2)=\bSi(\xi_1,\xi_2),
\]
called the \emph{covariance form}. The   random vectors $X_1$ and $X_2$ are independent if and only if  they are uncorrelated, i.e.,
\[
\cov(X_1,X_2)=0.
\]
We can then identify  $\cov(X_1,X_2)$ with a linear operator $\Cov(X_1,X_2):\bsV_2\ra \bsV_1$, via the equality
\[
\begin{split}
\bsE\bigl(\,\lan \xi_1,X_1\ran\lan\xi_2, X_2\ran\,\bigr) &=\cov(X_1,X_2)(\xi_1,\xi_2)\\
&=\bigl\lan\, \xi_1, \Cov(X_1,X_2)\xi_2^\dag\,\bigr\ran,\;\;\forall \xi_1\in\bsV_1\dual,\;\;\xi_2\in\bsV_2\dual,
\end{split}
\]
where $\xi_2^\dag\in\bsV_2$ denotes the vector metric dual to $\xi_2$. The  operator $\Cov(X_1,X_2)$ is called the \emph{covariance operator} of $X_1, X_2$.  

The  conditional random variable $(X_1|X_2=x_2)$  has probability density 
\[
p_{(X_1|X_2=\bx_2)}(\bx_1)=  \frac{p_{X_1\oplus X_2}(\bx_1,\bx_2)}{\int_{\bsV_1} p_{X_1\oplus X_2}(\bx_1,\bx_2) |d\bx_1|}.
\]
For a measurable function $f:\bsV_1\ra \bR$ the conditional expectation   $\bsE(f(X_1)|X_2=\bx_2)$ is the (deterministic) scalar
\[
\bsE(f(X_1)|X_2=\bx_2)= \int_{\bsV_1} f(\bx_1) p_{(X_1|X_2=\bx_2)}(\bx_1)|d\bx_1|.
\]
If $X_2$ is nondegenerate,   the  \emph{regression formula}, \cite{AzWs}, implies that  the random vector   $(X_1|X_2=x_2)$ is a Gaussian  vector with covariance operator
\begin{equation}
\bsS(X_1|X_2=x_2)=\bsS(X_1)-\Cov(X_1,X_2)\bsS(X_2)^{-1} \Cov(X_2,X_1),
\label{eq: cov-regr}
\end{equation}
and  expectation  
\begin{equation}
\bsE(X_1|X_2=x_2)=Cx_2,
\label{eq: cond-expect}
\end{equation}
 where  $C$ is given by
\begin{equation}
 C= \Cov(X_1,X_2)\bsS(X_2)^{-1}.
\label{eq: gauss-c}
\end{equation}

\section{A class of random symmetric matrices}
\label{s: gmat}
\setcounter{equation}{0}

We denote by $\eS_m$ the space of real symmetric   $m\times m$  matrices. This is an Euclidean space with respect to the inner product $(A,B):=\tr(AB)$. This inner product is invariant with respect to the action of $\SO(m)$ on $\eS_m$. We set
\[
\widehat{\bsE}_{ij}:=\begin{cases}
\bsE_{ij}, & i=j\\
\frac{1}{\sqrt{2}}E_{ij}, & i<j.
\end{cases}.
\]
The collection  $(\widehat{\bsE}_{ij})_{i\leq j}$ is a basis  of $\eS_m$ orthonormal with respect to the above inner product.  We set
\[
\hat{a}_{ij}:= \begin{cases}
a_{ij}, & i=j\\
\sqrt{2}a_{ij}, & i<j.
\end{cases}
\]
The collection $(\hat{a}_{ij})_{i\leq j}$  the orthonormal basis of $\eS_m\dual$ dual to $(\widehat{\bsE}_{ij})$.  The volume density induced by this metric is
\[
|dA|:=\prod_{i\leq j} d\widehat{a}_{ij}= 2^{\frac{1}{2}\binom{m}{2}}\prod_{i\leq j} da_{ij}.
\]
Throughout  the paper  we     encountered a $2$-parameter family of Gaussian    probability measures  on $\eS_m$.  More precisely for any    real numbers  $u,v$ such that
\[
v>0, mu+2v>0,
\]
 we  denote by $\eS_m^{u,v}$  the space  $\eS_m$     equipped with the centered Gaussian measure $d\bGamma_{u,v}(A)$ uniquely determined by the covariance equalities
 \[
 \bsE(a_{ij}a_{k\ell})=  u\delta_{ij}\delta_{k\ell}+ v(\delta_{ik}\delta_{j\ell}+ \delta_{i\ell}\delta_{jk}),\;\;\forall 1\leq i,j,.k,\ell\leq m.
 \]
 In particular we have
 \[
 \bsE(a_{ii}^2)= u+2v,\;\;\bsE(a_{ii}a_{jj})=u,\;\;\;\bsE(a_{ij}^2)=v,\;\;\forall 1\leq i\neq j\leq m,
 \]
while all other covariances are trivial.  The  ensemble  $\eS_m^{0,v}$ is    a rescaled version of  the Gaussian Orthogonal Ensemble  (GOE) and we will refer to it as $\GOE_m^v$.       

For $u>0$ the ensemble $\eS_m^{u,v}$ can be given an alternate description.   More precisely   a random $A\in \eS_m^{u,v}$ can be described as a sum
\[
A= B+ \ X\one_m,\;\;B\in \GOE_m^v,\;\; X\in \bsN(0, u),\;\;\mbox{ $B$ and $X$ independent}.
\]
We write  this
\begin{equation}
\eS_m^{u,v} =\GOE_m^v\hat{+}\bsN(0,u)\one_m,
\label{eq: smgoe}
\end{equation}
where $\hat{+}$ indicates a sum of \emph{independent} variables.

The  Gaussian measure $d\bGamma_{u,v}$ coincides with the Gaussian measure $d\bGamma_{u+2v,u,v}$ defined in \cite[App. B]{N2}.  We recall a few facts from   \cite[App. B]{N2}. 

The  probability density  $d\bGamma_{u,v}$  has the explicit description
\[
d\bGamma_{u,v}(A)= \frac{1}{(2\pi)^{\frac{m(m+1)}{4}}  \sqrt{D(u,v)}} e^{-\frac{1}{4v}\tr A^2-\frac{u'}{2}(\tr A)^2} |dA|,
\]
 where
 \[
 D(u,v)= (2v)^{(m-1)+\binom{m}{2}}\bigl( mu +2v\,\bigr),
 \]
 and 
 \[
 u'=\frac{1}{m}\left(\frac{1}{mu+2v}-\frac{1}{2v}\right)=-\frac{u}{2v(mu+2v)}.
  \]
 In the special case  $\GOE_m^v$ we have $u=u'=0$  and  
\begin{equation}
d\bGamma_{0,v}(A)=\frac{1}{(2\pi v)^{\frac{m(m+1)}{4}} } e^{-\frac{1}{4v}\tr A^2} |dA|.
\label{eq: gov}
\end{equation}
We have a   \emph{Weyl integration formula} \cite {AGZ} which states that if  $f: \eS_m\ra \bR$ is a measurable  function which  is invariant under  conjugation, then the     value $f(A)$ at $A\in\eS_m$ depends only on the eigenvalues $\lambda_1(A)\leq \cdots \leq \lambda_n(A)$ of $A$ and we  have
\begin{equation}
\bsE_{\GOE_m^v}\bigl(\,f(X)\,\bigr)=\frac{1}{\bsZ_m(v)} \int_{\bR^m}  f(\lambda_1,\dotsc ,\lambda_m) \underbrace{\left(\prod_{1\leq i< j\leq m}|\lambda_i-\lambda_j| \right)\prod_{i=1}^m e^{-\frac{\lambda_i^2}{4v}} }_{=:Q_{m,v}(\lambda)} |d\lambda_1\cdots d\lambda_m|,
\label{eq: weyl}
\end{equation}
where the normalization constant $\bsZ_m(v)$ is defined  by
\[
{\bsZ_m(v)} =\int_{\bR^m}   \prod_{1\leq i< j\leq m}|\lambda_i-\lambda_j| \prod_{i=1}^m e^{-\frac{\lambda_i^2}{4v}} |d\lambda_1\cdots d\lambda_m|
\]
\[
=(2v)^{\frac{m(m+1)}{4}} \underbrace{\int_{\bR^m}   \prod_{1\leq i< j\leq m}|\lambda_i-\lambda_j| \prod_{i=1}^m e^{-\frac{\lambda_i^2}{2}} |d\lambda_1\cdots d\lambda_m|}_{=:\bsZ_m}.
\]
The precise value of $\bsZ_m$  can be computed via Selberg integrals, \cite[Eq. (2.5.11)]{AGZ}, and we have
\begin{equation}
\bsZ_m=(2\pi)^{\frac{m}{2}} m!\prod_{j=1}^{m}\frac{\Gamma(\frac{j}{2})}{\Gamma(\frac{1}{2})}=2^{\frac{m}{2}}m!\prod_{j=1}^m \Gamma\left(\frac{j}{2}\right).
\label{eq: zm}
\end{equation}
For any positive integer $n$ we define the \emph{normalized} $1$-point corelation function $\rho_{n,v}(x)$ of $\GOE_n^v$ to be
\[
\rho_{n,v}(x)= \frac{1}{\bsZ_n(v)}\int_{\bR^{n-1}} Q_{n,v}(x,\lambda_2,\dotsc, \lambda_n) d\lambda_1\cdots d\lambda_n.
\]
For any Borel measurable function $f:\bR\ra \bR$  we have \cite[\S 4.4]{DG}
 \begin{equation}
 \frac{1}{n}\bsE_{\GOE_n^v} \bigl(\,\tr f(X)\,\bigr)= \int_{\bR} f(\lambda) \rho_{n,v}(\lambda) d\lambda. 
 \label{eq: 1pcor}
 \end{equation}
The equality (\ref{eq: 1pcor}) characterizes $\rho_{n,v}$. Let us  observe that for any constant $c>0$, if 
\[
A\in \GOE_n^v\Llra cA\in \GOE_n^{c^2v}.
\]
Hence for any Borel set $B\subset \bR$   we have
\[
\int_{cB} \rho_{n,c^2v}(x) dx =\int_B \rho_{n,v} (y) dy.
\]
 We conclude that
 \begin{equation}
 c\rho_{n,c^2v}(cy)= \rho_{n,v}(y),\;\;\forall n,c,y.
 \label{eq: resc-cor}
 \end{equation}
 The   behavior of the $1$-point correlation  function $\rho_{n,v}(x)$ for $n$ large  is described by \emph{Wigner semicircle law}  which states that for any $v>0$ the sequence  of measures on $\bR$
\[
\rho_{n,vn^{-1}}(x) dx =n^{\frac{1}{2}}\rho_{n,v}(n^{\frac{1}{2}} x) dx
\]
converges weakly  as $n\ra \infty$ to the  semicircle distribution
\[
\rho_{\infty,v}(x)|dx|= \bsI_{\{|x|\leq 2\sqrt{v}\}}\frac{1}{2\pi v}\sqrt{4v-x^2} |dx|.
\]
The   expected value of the absolute value of the  determinant of     of a  random $A\in \GOE_m^v$ can be expressed neatly in terms of the correlation function $\rho_{m+1,v}$.   More precisely, we have the following result first observed by Y.V. Fyodorov \cite{Fy} in a context related to ours.

\begin{lemma}  Suppose $v>0$. Then for any $c\in\bR$ we have
\[
\bsE_{\GOE_m^v} \bigl(\,|\det(A-c\one_m)|\,\bigr)= 2^{\frac{3}{2}}(2v)^{\frac{m+1}{2}}\Gamma\left(\frac{m+3}{2}\right) e^{\frac{c^2}{4v}}\rho_{m+1,v}(c) .
\]
\label{lemma: exp-det}
\end{lemma}

\begin{proof}Using the Weyl integration formula we deduce
\[
\bsE_{\GOE_m^v} \bigl(\,|\det(A-c\one_m)|\,\bigr) =\frac{1}{\bsZ_m(v)}\int_{\bR^m} \prod_{i=1}^me^{-\frac{\lambda_i^2}{4v}}|c-\lambda_i|\prod_{i\leq j}|\lambda_i-\lambda_j| d\lambda_1\cdots d\lambda_m
\]
\[
=\frac{e^{\frac{c^2}{4v}}}{\bsZ_m(v)} \int_{\bR^m} e^{-\frac{c^2}{4v}}\prod_{i=1}^me^{-\frac{\lambda_i^2}{4v}}|c-\lambda_i|\prod_{i\leq j}|\lambda_i-\lambda_j| d\lambda_1\cdots d\lambda_m
\]
\[
= \frac{e^{\frac{c^2}{4v}}\bsZ_{m+1}(v)}{\bsZ_m(v)}\frac{1}{\bsZ_{m+1}(v)} \int_{\bR^m} Q_{m+1,v}(c,\lambda_1,\dotsc, \lambda_m) d\lambda_1\cdots d\lambda_m
\]
\[
=\frac{e^{\frac{c^2}{4v}}\bsZ_{m+1}(v)}{\bsZ_m(v)} \rho_{m+1,v}(c)=v^{\frac{m+1}{2}}\frac{e^{\frac{c^2}{4v}}\bsZ_{m+1}}{\bsZ_m} \rho_{m+1,v}(c)
\]
\[
=(m+1)\sqrt{2}(2v)^{\frac{m+1}{2}}e^{\frac{c^2}{4v}}\Gamma\left(\frac{m+1}{2}\right) \rho_{m+1,v}(c)= 2^{\frac{3}{2}}(2v)^{\frac{m+1}{2}}\Gamma\left(\frac{m+3}{2}\right) e^{\frac{c^2}{4v}} \rho_{m+1,v}(c) .
\]
\end{proof}

The above result admits the following generalization, \cite[Lemma 3.2.3]{Auff0}.

\begin{lemma} Let $u>0$. Then
\[
\bsE_{\eS_m^{u,v}}\bigl(\,|\det (A-c\one_m)|\,\bigr)= 2^{\frac{3}{2}}(2v)^{\frac{m+1}{2}}\Gamma\left(\frac{m+3}{2}\right) \frac{1}{\sqrt{2\pi u}}\int_{\bR}\rho_{m+1,v}(c-x) e^{\frac{(c-x)^2}{4v}-\frac{x^2}{2u}} dx.
\]
In particular, if  $u=2kv$, $k<1$ we have
\[
\bsE_{\eS_m^{2kv,v}}\bigl(\,|\det (A-c\one_m)|\,\bigr)= 2^{\frac{3}{2}}(2v)^{\frac{m}{2}}\Gamma\left(\frac{m+3}{2}\right) \frac{1}{\sqrt{2\pi k}}\int_{\bR}\rho_{m+1,v}(c-x) e^{ -\frac{1}{4vt_k^2}(x + t_k^2c)^2 +\frac{(t_k^2+1)c^2}{4v}} dx,
\]
($\lambda:=c-x$)
\[
= 2^{\frac{3}{2}}(2v)^{\frac{m}{2}}\Gamma\left(\frac{m+3}{2}\right) \frac{1}{\sqrt{2\pi k}}\int_{\bR}\rho_{m+1,v}(\lambda) e^{ -\frac{1}{4vt_k^2}(\lambda -(t_k^2+1)c)^2 +\frac{(t_k^2-1)c^2}{4v}} d\lambda,
\]
where
\[
t_k^2:=\frac{1}{\frac{1}{k}-1}=\frac{k}{1-k}.
\]
\label{lemma: exp-det1}
\end{lemma}

\begin{proof}   Recall the equality (\ref{eq: smgoe})  $\eS_m^{u,v} =\GOE_m^v\hat{+}\bsN(0,u)\one_m$. We deduce that
\[
\bsE_{\eS_m^{u,v}}\bigl(\,|\det (A-c\one_m)|\,\bigr)=\bsE\bigl(\,\det(B+(X-c)\one)|\,\bigr)
\]
\[
=\frac{1}{\sqrt{2\pi u}}\int_{\bR}\bsE_{\GOE_m^v}\bigl(\, |\det (B-(c-X)\one_m)|\;\bigl|\; X=x)  e^{-\frac{x^2}{2u}} dx
\]
\[
=\frac{1}{\sqrt{2\pi u}}\int_{\bR}\bsE_{\GOE_m^v}\bigl(\, |\det (B-(c-x)\one_m)|\,\bigr)  e^{-\frac{x^2}{2u}} dx
\]
\[
= 2^{\frac{3}{2}}(2v)^{\frac{m+1}{2}}\Gamma\left(\frac{m+3}{2}\right) \frac{1}{\sqrt{2\pi u}}\int_{\bR}\rho_{m+1,v}(c-x) e^{\frac{(c-x)^2}{4v}-\frac{x^2}{2u}} dx.
\]
Now observe that if $ u=2k v$ then
\[
\frac{(c-x)^2}{4v}-\frac{x^2}{2u}=  -\frac{x^2}{4k v}+\frac{1}{4v}(x^2 -2cx+ c^2)
\]
\[
=\frac{1}{4v} \left( -\frac{1}{t_k^2}x^2-2cx -c^2t_k^2\right) + \frac{c^2(1+t_k^2)}{4v}= -\frac{1}{4vt_k^2}(x+t_k^2c)^2+ \frac{c^2(1+t_k^2)}{4v}.
\]

\end{proof}

\end{document}